\documentclass[12pt,reqno]{amsart}
\usepackage[T1]{fontenc}
\usepackage[english]{babel}
\usepackage[utf8]{inputenc}
\usepackage[backend=bibtex,doi=false,giveninits=true,isbn=false,url=false]{biblatex}
\usepackage{amssymb}
\usepackage{amsmath}
\usepackage{xcolor}
\usepackage{amsfonts}
\usepackage{amsthm}
\usepackage{csquotes}
\usepackage[a4paper,margin=2.5cm]{geometry}

\usepackage[colorlinks]{hyperref}
\theoremstyle{plain}

\newcommand{\esp}{\operatorname{\mathbb{E}}}
\newcommand{\prob}{\mathbb{P}}
\newcommand{\norm}[1]{\|#1\|}
\newcommand{\real}{\mathbb{R}}

\title[Wishart-type matrix estimation]{Almost sharp covariance and Wishart-type matrix estimation}

\author[P.\ O.\ Santos]{Patrick Oliveira Santos$^1$\\
}
\thanks{$^1$ LAMA, Université Gustave Eiffel, Paris, France. Email: patrick.oliveirasantos@u-pem.fr}

\begin{document}

\newtheorem{thm}{Theorem}[section]

\newtheorem{coro}[thm]{Corollary}
\newtheorem{lemma}[thm]{Lemma}
\newtheorem{conj}[thm]{Conjecture}
\newtheorem{propo}[thm]{Proposition}

\theoremstyle{definition}
\newtheorem{exa}[thm]{Example}
\newtheorem{defi}[thm]{Definition}
\newtheorem{remark}[thm]{Remark}

\maketitle

\begin{abstract} 
Let $X_1,..., X_n \in \real^d$ be independent Gaussian random vectors with independent entries and variance profile $(b_{ij})_{i \in [d],j \in [n]}$. A major question in the study of covariance estimation is to give precise control on the deviation of $\sum_{j \in [n]}X_jX_j^T-\esp X_jX_j^T$. In this paper, we improve the results in \cite{cai2022non,bandeira2021matrix} and we show that under mild conditions, we have
\begin{align*}
    \esp \left\|\sum_{j \in [n]}X_jX_j^T-\esp X_jX_j^T\right\| \lesssim  \max_{i \in [d]}\left(\sum_{j \in [n]}\sum_{l \in [d]}b_{ij}^2b_{lj}^2\right)^{1/2}+\max_{j \in [n]}\sum_{i \in [d]}b_{ij}^2+\text{error}.
\end{align*}
The error is quantifiable, and we often capture the $4$th-moment dependency already presented in \cite{cai2022non} for some examples. The proofs are based on the moment method and a careful analysis of the structure of the shapes that matter. We also provide examples showing improvement over the past works and matching lower bounds.
\noindent \end{abstract}


\section{Introduction}\label{section intro}

The study of the norm of random matrices has increased significantly over the years, and bounding the operator norm has been proved one central topic in the field \cite{bai1988necessary,van2017spectral,bandeira2016sharp}. Particularly, several applications coming from statistics require a precise sharp control on the deviations of the empirical covariance problem \cite{koltchinskii2017concentration,minasyan2023statistically,zhivotovskiy2021dimension}. For instance, it is well-known \cite{vershynin2018high} that an i.i.d sample $X_1,...,X_n \in \real^d$ of isotropic Gaussian random vectors satisfies the following deviation
\begin{align}\label{general dimension dependent result}
    \esp\left\|\frac{1}{n}\sum_{j \in [n]}X_jX_j^T-\esp X_1X_1^T\right\| \lesssim \frac{d}{n} \vee \sqrt{\frac{d}{n}}.
\end{align}
Much less is known, however, when the identically distributed condition is removed and we only require independence. Our contribution comes precisely in this direction. In particular, we improve the results in \cite{cai2022non} and we shed light on the $4$th-moment parameter and its graph interpretation that was before unclear.

Let $X$ be a random $d\times n$ Gaussian matrix with independent entries $X_{ij}=b_{ij}g_{ij}$, where $b_{ij} \ge 0$ and $\{g_{ij}:i \in [d],j \in [n]\}$ are independent standard Gaussian random variables $N(0,1)$. Our goal is to bound the quantity
\begin{align*}
    \esp\norm{XX^T-\esp XX^T}=\esp \left\|\sum_{j \in [n]}X_jX_j^T-\esp X_jX_j^T\right\|,
\end{align*}
where $X_j=Xe_j$ is the $j$th column of $X$. One of the first dimension-free results improving bound \eqref{general dimension dependent result} was given in the i.i.d setting in \cite{koltchinskii2017concentration}. Their result states that whenever $Y_1,...,Y_n$ are i.i.d Gaussian random vectors in $\real^d$, we have
\begin{align*}
   \esp \norm{YY^T-\esp YY^T} \asymp \norm{\Sigma}\left(\sqrt{n\text{rk}(\Sigma)},\text{rk}(\Sigma)\right),
\end{align*}
where
\begin{align*}
    \text{rk}(\Sigma)=\frac{\text{tr}(\Sigma)}{\norm{\Sigma}}
\end{align*}
is the effective rank of the covariance matrix $\Sigma=\esp Y_1Y_1^T$. The dependency on the sample size $n$ is sharp in all i.i.d cases, but much less is evident when the vectors are not identically distributed. 

In an orthogonal direction, Bandeira and van Handel \cite{bandeira2016sharp} proved that
\begin{align*}
    \esp \norm{X} \lesssim \sigma_C+\sigma_R+C\sigma_* \sqrt{\log (n \wedge d)},
\end{align*}
where $\sigma_C$ is the maximum Euclidean norm of columns of $B=(b_{ij})$, $\sigma_R$ is the maximum Euclidean norm of rows of $B$ and $\sigma_*$ is the maximum entry of $B$, that is,
\begin{align*}
    &\sigma_C^2=\max_{j \in [n]}\sum_{i\in [d]}b_{ij}^2;\\
    &\sigma_R^2=\max_{i \in [d]}\sum_{j \in [n]} b_{ij}^2;\\
    &\sigma_*=\max_{(i,j) \in [d]\times [n]}|b_{ij}|.
\end{align*}
To prove such a result, they compared the moments $\esp \text{tr}(XX^T)^p$ to the moments of a standard Gaussian matrix $\esp \text{tr}(GG^T)^p$ with reduced dimensions. This comparison method turned out to be also efficient to prove the estimations for the covariance problem as well. In \cite{cai2022non}, T. Cai, Han and Zhang applied these techniques to $XX^T-\esp XX^T$ and they proved that

\begin{align*}
    \esp \norm{XX^T-\esp XX^T} \lesssim \sigma_C\sigma_R+\sigma_C^2+C(\sigma_C\sigma_*+\sigma_R\sigma_*)\sqrt{\log (n \wedge d)}+C'\sigma_*^2\log (n\wedge d).
\end{align*}
The leading term $\sigma_C\sigma_R+\sigma_C^2$ is not always sharp. Indeed, studying the case $b_{ij}=b_{j}$, that is, the rows are i.i.d, the authors of \cite{cai2022non} proved that
\begin{align*}
    \esp \norm{XX^T-\esp XX^T} \asymp \sqrt{d\sum_j b_{j}^4}+d\max_jb_j^2=\sqrt{d\sum_j b_{j}^4}+\sigma_C^2.
\end{align*}
Our main contribution shed light on this $4$th-moment parameter and how it appears from the moment method.

We begin our results for the operator norm. Define the parameters:
\begin{align*}
    &\bullet \tilde{\sigma}_\infty^2=\max_{i,l:i \ne l}\sum_{j \in [n]}b_{ij}^2b_{lj}^2;\\
    &\bullet \bar{\sigma}_\infty^2=\max_{i \in [d]}\sum_{j \in [n]}b_{ij}^4;\\
    &\bullet\sigma_\infty^2=\max_{i \in [d]}\sum_{j \in [n]}\sum_{l:l\ne i}b_{ij}^2b_{lj}^2;\\
    &\bullet\beta_\infty=\frac{\tilde{\sigma}_\infty \sigma_C}{\sigma_\infty \sigma_*}.
\end{align*}
Notice in particular that $\tilde{\sigma}_\infty \le \bar{\sigma}_\infty$, by Cauchy-Schwarz inequality.
\begin{thm}\label{main theorem}
Let $X$ be a $d \times n$ Gaussian matrix with independent entries such that $X_{ij}=b_{ij}g_{ij}$ where $\{g_{ij}: (i,j) \in [d]\times [n]\}$ are i.i.d standard Gaussian r.v. Then, if $\beta_\infty \le 1$, we have
\begin{align*}
    &\esp \norm{XX^T-\esp XX^T}=\esp \left\|\sum_{j \in [n]}X_jX_j^T-\esp X_jX_j^T\right\| \\
    &\le (1+\varepsilon)\left\{
    2\sigma_\infty+\sigma_C^2+C(\varepsilon)\sigma_*\left(\sigma_C+\frac{\sigma_\infty}{\sigma_C}\right)\sqrt{\log(n \wedge d)}+C^2(\varepsilon)\sigma_*^2 \log (n\wedge d)
    \right\},
\end{align*}
for any $0<\varepsilon \le 1/2$. Otherwise, $\beta_\infty>1$ and we have
\begin{align*}
    &\esp \norm{XX^T-\esp XX^T} \\
    &\le (1+\varepsilon)\left\{
    \frac{2\tilde{\sigma}_\infty \sigma_C}{\sigma_*}+\sigma_C^2+C(\varepsilon)\left(\sigma_C\sigma_*+\bar{\sigma}_\infty\right)\sqrt{\log(n\wedge d)}+C^2(\varepsilon)\sigma_*^2 \log (n\wedge d)
    \right\}.
\end{align*}
The constant $C(\varepsilon)$ is
\begin{align*}
    C(\varepsilon)=\frac{C(1+\varepsilon)}{\sqrt{\log (1+\varepsilon)}},
\end{align*}
where $C$ is a universal constant.
\end{thm}

Theorem \ref{main theorem} improves Theorem 2.1 from \cite{cai2022non}. When $\beta_\infty \le 1$, we get the sharp constant $2\sigma_\infty$ on the right-hand side. Moreover, if $\beta_\infty=O(1)$, our result still gives the correct order of magnitude. 


Since the method of proof uses the moment method, we can extend Theorem \ref{main theorem} to estimate Schatten norms. Recall that the $p$-Schatten norm is defined by
\begin{align*}
    \norm{A}_{S_p}^p=\text{Tr}(A)^p,
\end{align*}
for a positive matrix $A$. It is also the same as the $p$-norm of the singular values of $A$. This time, we define the more involved parameters
\begin{align*}
&\bullet \sigma_p=\left\{\sum_{i \in [d]}\left[\sum_{j \in [n]}\sum_{l \in [d]}b_{ij}^2b_{lj}^2\right]^{p/2}\right\}^{1/p};\\
&\bullet \overline{\sigma}_p=\left\{\sum_{i \in [d]}\left[\sum_{j \in [n]}b_{ij}^4\right]^{p/2}\right\}^{1/p};\\
&\bullet b_p=\left\{\sum_{i \in [d]}\max_{j \in [n]}b_{ij}^{2p}\right\}^{1/(2p)};\\
&\bullet \beta_p=\frac{\bar{\sigma}_p\sigma_C}{\sigma_p b_p}.
\end{align*}

Our second main theorem is the following.
\begin{thm}\label{thm second main}
Let $X$ be a $d \times n$ Gaussian matrix with independent entries such that $X_{ij}=b_{ij}g_{ij}$ where $\{g_{ij}: (i,j) \in [d]\times [n]\}$ are i.i.d standard Gaussian r.v. Then,  if $\beta_p \le 1$, we have
\begin{align*}
(\esp \emph{Tr}[XX^T-\esp XX^T]^p)^{1/p} \le d^{1/p}\left\{2\sigma_p+\sigma_C^2+C\sqrt{p}\left(\sigma_C\sigma_*+\frac{\sigma_p\sigma_*}{\sigma_C}\right)+C'pb_p^2\right\}.
\end{align*}
Otherwise, $\beta_p>1$ and
\begin{align*}
 (\esp \emph{Tr}[XX^T-\esp XX^T]^p)^{1/p} \le d^{1/p}\left\{\frac{2\overline{\sigma}_p\sigma_C}{\sigma_*}+\sigma_C^2+C\sqrt{p}\left(\sigma_C\sigma_*+\overline{\sigma}_p\right)+ C'pb_p^2\right\}.
\end{align*}
\end{thm}


\subsection{Main ideas of the proof}

The proof relies on the moment method and a careful analysis of paths. We will first remove the diagonal $\text{Diag}(XX^T)$, so that
\begin{align*}
    \esp \norm{XX^T-\esp XX^T} \le \esp \norm{\Delta XX^T}+\esp \norm{\text{Diag}(XX^T)-\esp XX^T},
\end{align*}
where $\Delta XX^T$ is the matrix of off-diagonal elements of $XX^T$. It turns out that the contribution of the diagonal is sufficiently small and can be added as an error factor (see Theorem \ref{thm diagonal}). On the other hand, the combinatorics of $\Delta XX^T$ are much easier to deal with. In particular, all paths in the complete bipartite graph over $[d] \sqcup [n]$ have all right vertices with at least two neighbors. 

We then proceed with the moment method. Note that 
\begin{align*}
    \esp \norm{Y} \le \esp \norm{Y}_{S_p} \le d^{1/p}\esp \norm{Y},
\end{align*}
for any symmetric $d \times d$ matrix. Hence,
\begin{align*}
    \esp \norm{Y} \le (\esp \norm{Y}_{S_p}^p)^{1/p},
\end{align*}
by Jensen's Inequality. We apply this for $Y=\Delta XX^T$ and our goal is to obtain a comparison lemma such as
\begin{align*}
    \esp \text{Tr}(\Delta XX^T)^p \le \kappa \esp \text{Tr}(\Delta GG^T)^p,
\end{align*}
where $\kappa>0$ and $G$ is a Gaussian matrix with reduced dimensions as in \cite{bandeira2016sharp}.

\subsection{Outline of the paper} The paper is organized as follows. In section \ref{section proofs}, we will provide the main proofs of theorems \ref{main theorem} and \ref{thm second main}. In section \ref{sec examples}, we will give examples to illustrate the improvement from the previous results. Finally, in section \ref{section lower bounds}, we will prove almost sharp matching lower bounds for our main theorems.

\bigskip
\textit{Notation.} Let us clarify some notation used throughout the paper. We denote $a \lesssim b$ or $a=O(b)$ if there exists an absolute constant $C$ such that $a \le Cb$. We also denote it as $b \gtrsim a$. If $a \lesssim b$ and $b \lesssim a$ hold, we denote $a \asymp b$. We write $a \wedge b= \min(a,b)$ and $a \vee b=\max(a,b)$. We denote $[n]=\{1,\ldots,n\}$ and $A \sqcup B$ is the disjoint union of two sets $A$ and $B$. Finally, we use $C,c, C',\ldots$ for universal numerical constants.

\paragraph*{\bfseries Acknowledgments.}
We thank Olivier Guédon for pointing out this problem and helpful discussions.



\section{Proofs}\label{section proofs}
\subsection{Preliminaries}\label{sub preliminaries}
We begin by recalling the Gaussian integration by parts lemma.
\begin{lemma}
    Let $g \sim N(0,1)$ be a standard Gaussian r.v. and $f \in C^1(\mathbb{R})$, then
\begin{align*}
    \esp gf(g)=\esp f'(g).
\end{align*}
\end{lemma}
The authors of \cite{cai2022non} deduced from this lemma a simple property of the joint moments of $g$ and $g^2-1$.
\begin{lemma}\label{joint moments Gaussians lemma}
    Let $a_{n,m}=\esp g^n(g^2-1)^m$, where $g \sim N(0,1)$. Then $a_{n,m} \ge 0$ and $a_{n,m}=0$ if and only if $n$ is odd or $(n,m)=(0,1)$.
\end{lemma}

We also recall the sharp bound on the operator norm for a standard Gaussian matrix shown in \cite{cai2022non}.
\begin{propo}\label{propo standard Gaussian bound}
    Let $G$ be a $d\times n$ Gaussian matrix with i.i.d standard Gaussian r.v. entries. Then, for any $p \ge 2$ we have
    \begin{align*}
        (\esp \norm{GG^T-\esp GG^T}^p)^{1/p} \le 2\sqrt{dn}+d+4\sqrt{p}(\sqrt{d}+\sqrt{n})+2p.
    \end{align*}
\end{propo}
Note that
\begin{align*}
    \esp \norm{\text{Diag}(GG^T)-\esp GG^T}^p=\esp \max_{i \in [d]}\left(\sum_{j \in [n]}(g^2_{ij}-1)\right)^p.
\end{align*}
Bernstein's Inequality \cite{chafai2012interactions} implies then that 
\begin{align*}
    ( \esp \norm{\text{Diag}(GG^T)-\esp GG^T}^p)^{1/p} \lesssim \sqrt{p n}+p.
\end{align*}
Consequently, we end this subsection with a corollary for the off-diagonal part.
\begin{coro}\label{coro norm of off-diagonal standard}
    Let $G$ be a $d\times n$ Gaussian matrix with i.i.d standard Gaussian entries. Then, for any $p \ge 2$ we have
    \begin{align*}
        (\esp \norm{\Delta (GG^T)}^p)^{1/p} \le 2\sqrt{dn}+d+C\sqrt{p}(\sqrt{d}+\sqrt{n})+C'p.
    \end{align*}
\end{coro}


\subsection{The diagonal part}\label{sub diagonal}
In this section, the main result is the following.
\begin{thm}\label{thm diagonal}
For any $p \ge 2$, we have
\begin{align*}
\left(\esp \emph{Tr}[\emph{Diag}(XX^T)-\esp XX^T]^p\right)^{1/p} \asymp \sqrt{p}\overline{\sigma}_p+pb_p^2.
\end{align*}
\end{thm}
\begin{proof}
For the upper bound, note that
\begin{align*}
\esp \text{Tr}(\text{Diag}(XX^T)-\esp XX^T)^p=\sum_{i \in [d]}\esp \left(\sum_{j \in [n]}b_{ij}^2(g_{ij}^2-1)\right)^p.
\end{align*}
Since $g_{ij}^2-1$ are independent, centered, and subexponential, we can use Bernstein's Inequality to deduce that
\begin{align*}
\mathbb{P}\left(\left|\sum_{j \in [n]}b_{ij}^2(g_{ij}^2-1)\right| \ge t\right) \le 2\exp\left(-c\min\left\{\frac{t^2}{a^2},\frac{t}{b}\right\}\right),
\end{align*}
where
\begin{align*}
&a=\sum_{j \in [n]}b_{ij}^4;\\
&b=\max_{j \in [n]}b_{ij}^2,
\end{align*}
hence
\begin{align*}
\left[\esp \left(\sum_{j \in [n]}b_{ij}^2(g_{ij}^2-1)\right)^p\right]^{1/p} \lesssim \sqrt{p} \left(\sum_{j \in [n]}b_{ij}^4\right)^{1/2}+p\max_{j \in [n]}b_{ij}^2.
\end{align*}
We then have
\begin{align*}
\left[\esp \text{Tr}(\text{Diag}(XX^T)-\esp XX^T)^p\right]^{1/p} &\lesssim \left[\sum_{i \in [d]}\left(\sqrt{p} \left(\sum_{j \in [n]}b_{ij}^4\right)^{1/2}+p\max_{j \in [n]}b_{ij}^2.\right)^p\right]^{1/p}\\
&\lesssim \sqrt{p}\overline{\sigma}_p+pb_p^2,
\end{align*}
where the last inequality follows by the triangle inequality.

For the lower bound, let $j_i$ be the index such that
\begin{align*}
\max_{j \in [n]}b_{ij}=b_{ij_i}.
\end{align*}
Since the joint moments of $g$ and $g^2-1$ are always positive, we deduce that
\begin{align*}
\left[\esp \text{Tr}(\text{Diag}(XX^T)-\esp XX^T)^p\right]^{1/p} \ge  \left[\sum_{i \in [d]}\esp b_{ij_i}^{2p}(g_{ij_i}^2-1)^p\right]^{1/p}.
\end{align*}
Now, the estimate
\begin{align*}
\left(\esp (g^2-1)^p\right)^{1/p} \gtrsim p,
\end{align*}
that follows the lower bound on the double factorial (see Lemma 5.2 in \cite{cai2022non}) implies that
\begin{align*}
\left[\esp \text{Tr}(\text{Diag}(XX^T)-\esp XX^T)^p\right]^{1/p} \gtrsim p\left(\sum_{i \in [d]}\max_{j \in [n]}b_{ij}^{2p}\right)^{1/p}=pb_p^2.
\end{align*}
On the other hand, Theorem 6 in \cite{zhang2020non} yields that $Z=\sum_{j \in [n]}b_{ij}^2(g_{ij}^2-1)$ satisfies
\begin{align*}
&\exp(-Ct^2/a)\lesssim \prob(Z \ge t);\\
&\exp(-Ct^2/a) \lesssim \prob(Z \le -t),
\end{align*}
for all $t \ge 0$. Therefore, its moments are lower bounded by the ones of the Gaussian $h \sim N(0, a)$, hence
\begin{align*}
(\esp |Z|^p)^{1/p} \gtrsim \sqrt{ap},
\end{align*}
so we conclude that
\begin{align*}
\left[\esp \text{Tr}(\text{Diag}(XX^T)-\esp XX^T)^p\right]^{1/p} \gtrsim \sqrt{p}\overline{\sigma}_p.
\end{align*}
\end{proof}


\subsection{The off-diagonal part}\label{sub off-diagonal}

The proof of the bounds for the off-diagonal part follows the moment method. First, we open the trace so that
\begin{align*}
\esp \text{Tr}(\Delta XX^T)^p&=\sum_{u \in [d]^p} \esp \prod_{k=1}^p (XX^T)_{u_ku_{k+1}} \mathbf{1}_{u_k \ne u_{k+1}}\\
&=\sum_{u \in [d]^p}\sum_{v \in [n]^p} \esp \prod_{k=1}^p X_{u_kv_k}X_{u_{k+1}v_k}\mathbf{1}_{u_k \ne u_{k+1}},
\end{align*}
where $u_{p+1}:=u_1$. We view the path $u_1 \to v_1 \to u_2 \to \cdots \to u_p \to v_p \to u_1$ as a cycle in the complete bipartite graph over $[d]^{(l)} \sqcup [n]^{(r)}$, where $(l)$ and $(r)$ indicate left and right vertices (we will remove the indexes if the context is clear). For a path $(u,v)$, we define its shape $s(u,v)$ as relabelling its vertices in order of appearance. For instance, the path
\begin{align*}
3 \to 2' \to 4 \to 1' \to 3 \to 1' \to 4 \to 5' \to 3 
\end{align*}
has shape
\begin{align*}
1 \to 1' \to 2 \to 2' \to 1 \to 2' \to 2 \to 3' \to 1.
\end{align*}
Note that each edge $u_kv_{k}$ and $u_{k+1}v_k$ must appear at least twice in the path $(u,v)$, by the independence of the Gaussian r.v. and symmetry. Call the shapes that satisfy this \textit{even}. Let then $\mathcal{S}$ be the set of even shapes $s=(u,v)$ such that $u_k \ne u_{k+1}$ for all $k=1,...,p$. Moreover, the product
\begin{align}
    L(s):=\esp \prod_{k=1}^p g_{u_kv_k}g_{u_{k+1}v_k}
\end{align}
only depends on the shape of $(u,v)$, therefore we have
\begin{align*}
\esp \text{Tr}(\Delta XX^T)^p=\sum_{s \in \mathcal{S}}L(s)\sum_{\substack{(u,v) \in [d]^p\times [n]^p\\ s(u,v)=s}} \prod_{k=1}^p b_{u_kv_k}b_{u_{k+1}v_k}.
\end{align*}
Let $(m_1,m_2)=(m_1(s),m_2(s))$ be the quantity of right and left vertices that appear in the shape $s$. The key proposition to prove Theorem \ref{main theorem} is to bound
\begin{align}\label{defi W(s)}
    W(s):=\sum_{\substack{(u,v) \in [d]^p\times [n]^p\\ s(u,v)=s}} \prod_{k=1}^p b_{u_kv_k}b_{u_{k+1}v_k}
\end{align}
according to the number of vertices visited by the path. 
\begin{propo}\label{propo for the shapes and operator norm}
    Assume $\sigma_*=1$. If $\beta_\infty \le 1$, we have
    \begin{align*}
        W(s) \le \left[d \left(\frac{\sigma_\infty}{\sigma_C}\right)^{2m_1} \sigma_C^{2(m_2-1)}\right] \wedge \left[n \left(\frac{\sigma_\infty}{\sigma_C}\right)^{2(m_1-1)} \sigma_C^{2m_2}\right].
    \end{align*}
    Otherwise, $\beta_\infty>1$ and we have
    \begin{align*}
        W(s) \le \left[d \tilde{\sigma}_\infty^{2m_1} \sigma_C^{2(m_2-1)}\right] \wedge \left[n \tilde{\sigma}_\infty^{2(m_1-1)} \sigma_C^{2m_2}\right].
    \end{align*}
\end{propo}

Let us prove Theorem \ref{main theorem} given Proposition \ref{propo for the shapes and operator norm}.
\begin{proof}[Proof of Theorem \ref{main theorem}.]
Assume $\sigma_*=1$ (by homogeneity) and $\beta_\infty \le 1$. Let
\begin{align*}
    &a:=\frac{\sigma_\infty}{\sigma_C};\\
    &b:=\sigma_C.
\end{align*}
Then, using the first bound on Proposition \ref{propo for the shapes and operator norm}, we have
\begin{align*}
    \esp \text{Tr}(\Delta XX^T)^p \le d\sum_{s \in \mathcal{S}}L(s) \left(\frac{\sigma_\infty}{\sigma_C}\right)^{2m_1}\sigma_C^{2(m_2-1)}.
\end{align*}
On the other hand, for a standard Gaussian $r_2 \times r_1$ matrix $G$, we have
\begin{align*}
    \esp \text{Tr}(\Delta GG^T)^p=\sum_{s \in \mathcal{S}}L(s) \frac{r_1!}{(r_1-m_1)!}\frac{r_2!}{(r_2-m_2)!},
\end{align*}
for any $r_1,r_2>p/2$ (see \cite{cai2022non}). In particular, if $r_1=\lceil a^2 \rceil+p/2$ and $r_2=\lceil b^2 \rceil+p/2$, we have 
\begin{align*}
    \frac{r_1!}{(r_1-m_1)!} \ge r_1\cdots (r_1-m_1+1)^{m_1} \ge a^{2m_1}, 
\end{align*}
and
\begin{align*}
    \frac{r_2!}{(r_2-m_2)!} \ge r_2 b^{2(m_2-1)}.
\end{align*}
Hence
\begin{align*}
    \esp \text{Tr}(\Delta XX^T)^p \le \frac{d}{r_2}\esp \text{Tr}(\Delta GG^T)^p \le d\esp \norm{\Delta GG^T}^p.
\end{align*}
Now we estimate the latter by Corollary \ref{coro norm of off-diagonal standard} so that
\begin{align*}
    (\esp \norm{\Delta GG^T}^p)^{1/p} \le 2\sqrt{r_1r_2}+r_2+C\sqrt{p}(\sqrt{r_1}+\sqrt{r_2})+C'p.
\end{align*}
Together with Theorem \ref{thm diagonal}, we deduce that
\begin{align*}
    \esp \norm{XX^T-\esp XX^T}\le d^{\frac{1}{p}}\left\{2\sigma_\infty+\sigma_C^2+C\sqrt{p}\left(\sigma_C+\frac{\sigma_\infty}{\sigma_C}+\bar{\sigma}_p\right)+Cpb_p^2\right\}.
\end{align*}
Choose $p=\lceil \alpha \log d\rceil $. Since $\beta_\infty \le 1$, we have that
\begin{align*}
    \bar{\sigma}_p \le d^{\frac{1}{p}}\bar{\sigma}_\infty \le d^{\frac{1}{p}}\frac{\sigma_\infty}{\sigma_C}.
\end{align*}
Moreover, $b_p \le d^{\frac{1}{2p}}b_\infty$, thus
\begin{align*}
    \esp \norm{XX^T-\esp XX^T} \le e^{\frac{1}{\alpha}}\left\{2\sigma_\infty+\sigma_C^2+Ce^{\frac{1}{\alpha}}\sqrt{\alpha \log d}\left(\sigma_C+\frac{\sigma_\infty}{\sigma_C}\right)+C\alpha e^{\frac{1}{2\alpha}}\log d \right\}.
\end{align*}
Finally, set $1+\varepsilon=e^{\frac{1}{\alpha}}$, hence
\begin{align*}
    \alpha=\frac{1}{\log(1+\varepsilon)},
\end{align*}
and we get
\begin{align*}
    \esp \norm{XX^T-\esp XX^T}\le (1+\varepsilon)\left\{2\sigma_\infty+\sigma_C^2+C(\varepsilon)\sqrt{\log d}\left(\sigma_C+\frac{\sigma_\infty}{\sigma_C}\right)+C^2(\varepsilon)\log d\right\}.
\end{align*}
This gives the upper bound with $\log d$. The second bound in Proposition \ref{propo for the shapes and operator norm} yields the general bound for $\beta_\infty \le 1$. The case $\beta_\infty>1$ follows similarly. Indeed, we now set $(a,b)$ to be
\begin{align*}
    &a=\tilde{\sigma}_\infty;\\
    &b=\sigma_C,
\end{align*}
and then the previous proof follows straightforwardly. 

    
\end{proof}

Now we prove Proposition \ref{propo for the shapes and operator norm}.
\begin{proof}[Proof of Proposition \ref{propo for the shapes and operator norm}]
To simplify the notation, for a graph $G$, we will denote $e \in G$ if an edge $e$ belongs to $E(G)$, $v \in G$ if $v \in V(G)$ and $G'=G\setminus\{v\}$ is the subgraph of $G$ induced by the vertices $V(G)\setminus \{v\}$.  We use a similar notation to $G\setminus\{e\}$ and an edge $e \in E(G)$.

Given a shape $s \in \mathcal{S}$, we define a bipartite graph $G$ over $[m_2] \sqcup [m_1]$ so that $E(G)=\{(u_kv_k):k \in [p]\}$. Here, $[m_2]$ denotes the left vertices and $[m_1]$ denotes the right vertices. Let $k_e$ be the number of times each edge $e \in E(G)$ is traversed by the shape $s$, then $\sum_{e}k_e=2p=|k|$. According to \eqref{defi W(s)}, we get an alternative expression for $W(s)$:
\begin{align*}
    W(s)=\sum_{w_1 \ne \cdots \ne w_{m_2}}\sum_{t_1 \ne \cdots \ne t_{m_1}} \prod_{e=ij \in E(G)}b_{w_it_j}^{k_e}=:W^k(G),
\end{align*}
where the notation $w_1 \ne \cdots \ne w_{m_2}$ means that all $w_k$ are different, similarly for $t_k$. Note that, by the assumption on $s \in \mathcal{S}$, every right vertex has at least 2 neighbors. Now, fix $u_1=w_1=z \in [d]$ and define the following first-time arrivals:
\begin{align*}
    &i_1(k):=\inf\{l: u_l=k\}; \ k=2,...,m_2;\\
    &i_2(k):=\inf\{l:v_l=k\}; \ k=1,...,m_1.
\end{align*}
Let also $e^{(1)}_k=u_{i_1(k)}v_{i_1(k)-1}$ and $e^{(2)}_k=u_{i_2(k)}v_{i_2(k)}$. Then all these $m_1+m_2-1$ edges are distinct, and the subgraph $H$ generated by them is a spanning tree of $G$. 

The crucial distinction to \cite{cai2022non} is that we want to preserve the property that every right vertex has at least two neighbors. Call this property $\mathcal{P}$. Let us divide in two cases whether this is true. 

\textit{Case I.} Suppose the tree $H$ satisfies property $\mathcal{P}$. Assume $v,v'$ are extreme right vertices, that is,
\begin{align*}
    d(v,v')=\max_{r,r' \in [m_1]} d(r,r').
\end{align*}
(In case $m_1=1$, the result is trivial). Then $v$ has exactly one neighbor $u \in [m_2]$ such that $|N(u)| \ge 2$ and it satisfies
\begin{align*}
    d(v',u)=d(v',v)-1,
\end{align*}
that is, the unique path from $v'$ to $v$ passes through $u$. Indeed, if there are two of such vertices $u,u'$ and $u'$ is connected to both $v$ and a different $v''$, we would have that
\begin{align*}
    d(v',v'')=d(v',v)+d(v,v'')=d(v',v)+2,
\end{align*}
which contradicts the maximal distance of $v$ and $v'$. Therefore, if $L(v)=\{u \in N(v): |N(u)|=1\}\cup \{v\}$ we have that the graph $H'=H\setminus L(v)$ is still a tree with the property $\mathcal{P}$. Without loss of generality, we can assume that $v=m_1$. Since $\sigma_*=1$ and $k_e \ge 2$ for all $e \in G$, we have
\begin{align*}
    W^k(G) &\le d\sum_{w_2 \ne \cdots \ne w_{m_2}}\sum_{t_1 \ne \cdots \ne t_{m_1}} \prod_{e=ij \in E(H)}b_{w_it_j}^2 \\
    &\le d\left(\sum_{w_2 \ne \cdots \ne w_{m_2}}\sum_{t_1 \ne \cdots \ne t_{m_1-1}}\prod_{e=ij \in E(H')}b_{w_it_j}^2\right)\max_{w \in [d]}\sum_{j \in [n]}b_{wj}^2 \left(\sum_{l \in [d]:l \ne w}b_{lj}^2\right)^{|N(m_1)|-1}.
\end{align*}
For the second term, we further estimate
\begin{align*}
   \max_{w \in [d]}\sum_{j \in [n]}b_{wj}^2 \left(\sum_{l \in [d]:l \ne w}b_{lj}^2\right)^{|N(m_1)|-1} \le \sigma_\infty ^2\sigma_C^{2(|N(m_1)|-2)}.
\end{align*}
We then proceed by induction over the right vertices as we did for $H$. Here, induction is justified as $H'$ is still in case I. In particular, that yields
\begin{align*}
    W^k(G) \le d\sigma_\infty^{2m_1}\sigma_C^{2\sum_{v \in [m_1]}(|N(v)|-2)}.
\end{align*}
Since $|E(H)|=\sum_{v \in [m_1]}|N(v)|=m_1+m_2-1$, we get that
\begin{align*}
    W^k(G) \le d \left(\frac{\sigma_\infty}{\sigma_C}\right)^{2m_1}\sigma_C^{2(m_2-1)}.
\end{align*}

\textit{Case II.} In case the tree $H$ does not satisfy property $\mathcal{P}$, we then add for each $v \in H \cap [m_1]$ with $|N(v)|=1$ in $H$ one extra edge $uv \in E(G)$ from $G$. This creates a graph $H'$ that is not a tree, but it satisfies property $\mathcal{P}$. 

Let 
\begin{align*}
    V=\{v \in H \cap [m_1]: |N(v)|=1 \text{ in }H\}.
\end{align*} 
Then for each $v \in V$ we have $|N(v)|=2$ in $H'$ and $v$ belongs to a cycle in $H'$. In particular, we can remove $v$ from $H'$ and $H''=H'\setminus \{v\}$ is still connected. Assume $v=m_1$, then we have
\begin{align*}
    W^k(G) \le d\left(\sum_{w_2 \ne \cdots \ne w_{m_2}}\sum_{t_1 \ne \cdots \ne t_{m_1-1}} \prod_{e=ij \in E(H'')}b_{w_it_j}^2\right)\max_{i \ne l \in [d]}\sum_{j \in [d]}b_{ij}^2b_{lj}^2.
\end{align*}
We deduce that
\begin{align*}
    W^k(G) \le d\tilde{\sigma}_\infty^2 \sum_{w_2 \ne \cdots \ne w_{m_2}}\sum_{t_1 \ne \cdots \ne t_{m_1-1}} \prod_{e=ij \in E(H'')}b_{w_it_j}^2.
\end{align*}
By induction, we have
\begin{align*}
    W^k(G) \le d\tilde{\sigma}_\infty^{2|V|}\sum_{w_2 \ne \cdots \ne w_{m_2}}\sum_{t_1 \ne \cdots \ne t_{m_1-|V|}} \prod_{e=ij \in E(H\setminus V)}b_{w_it_j}^2.
\end{align*}
By assumption, $H\setminus V=H'\setminus V$ is a tree satisfying property $\mathcal{P}$. Therefore, case $I$ implies that
\begin{align*}
    W^k(G) \le d \tilde{\sigma}_\infty^{2|V|}\left(\frac{\sigma_\infty}{\sigma_C}\right)^{2(m_1-|V|)}\sigma_C^{2(m_2-1)}.
\end{align*}
By definition of $\beta_\infty$, we have
\begin{align*}
    W^k(G) \le d\beta_\infty^{2|V|} \left(\frac{\sigma_\infty}{\sigma_C}\right)^{2m_1} \sigma_C^{2(m_2-1)}.
\end{align*}
If $\beta_\infty \le 1$, we choose $|V|=0$, otherwise we choose $|V|=m_1$. A straightforward computation yields the bounds of Proposition \ref{propo for the shapes and operator norm} with factor $d$.

For the second bound, instead of fixing $u_1=w_1=z$, we fix $v_1=t_1=z$. Define the following first-time arrivals:
\begin{align*}
    &i_1(k):=\inf\{l: u_l=k\}; \ k=1,...,m_2;\\
    &i_2(k):=\inf\{l:v_l=k\}; \ k=2,...,m_1,
\end{align*}
and let also $e^{(1)}_k=u_{i_1(k)}v_{i_1(k)-1}$ and $e^{(2)}_k=u_{i_2(k)}v_{i_2(k)}$. The same argument done before implies that these $m_1+m_2-1$ edges are distinct, and the subgraph $H$ generated by them is a spanning tree of $G$. We then repeat the proof as in the first bound, but now the first choice of vertex $v_1$ will contribute with a factor of $n$.
\end{proof}

\subsection{Proof of Theorem \ref{thm second main}}

To get the correct parameters for the Schatten norm, we must improve Proposition \ref{propo for the shapes and operator norm} and the bound on $W(s)$. The main proposition of this subsection is the following.
\begin{propo}\label{propo for the shapes and the schatten norm}
For any shape $s \in \mathcal{S}$, if $\beta_p \le 1$, we have
\begin{align*}
W(s)\le d\sigma_*^{2p}\left\{\frac{\sigma_p}{\sigma_*\sigma_C}\right\}^{2m_1(s)}\left\{\frac{\sigma_C}{\sigma_*}\right\}^{2(m_2(s)-1)}.
\end{align*}
Otherwise $\beta_p>1$ and
\begin{align*}
W(s) \le d\sigma_*^{2p}\left\{\frac{\overline{\sigma}_p}{\sigma_*^2}\right\}^{2m_1(s)}\left\{\frac{\sigma_C}{\sigma_*}\right\}^{2(m_2(s)-1)}.   
\end{align*}
\end{propo}
As soon as Proposition \ref{propo for the shapes and the schatten norm} is available, the proof of Theorem \ref{thm second main} follows similarly as the proof of Theorem \ref{main theorem} and the bound for the diagonal in Theorem \ref{thm diagonal}. 

Proposition \ref{propo for the shapes and the schatten norm} follows the same argument shown in \cite{latala2018dimension}. On the other hand, we did not try to optimize the argument to our setting, instead, we prefer to prove it directly.

We start by the reduction to tree argument done in \cite{latala2018dimension} for $W^k(G)$. In this case, however, we want to keep track of the exponents for each right leaf that appears in the final reduction. We hence present the proof for completeness.
\begin{lemma}\label{lemma reduction to trees}
    Let $G$ be a graph generated by a shape $s \in \mathcal{S}$ and $k_e \ge 2$ for each $e \in E(G)$. Then, there exist $k_2',...,k_{m_1+m_2-1}' \ge 2$ such that $\sum_i k_i'=\sum_e k_e$ and
    \begin{align*}
        W^k(G) \le \max_{T \in \emph{span}(G)} W^{k'}(T), 
    \end{align*}
    where $\emph{span}(G)$ is the set of spanning trees of $G$. Moreover, the maximum can be taken such that whenever $T$ has a right leaf $v \in [m_1]$ with unique edge $e=uv \in T$ we have $k_e \ge 4$.
\end{lemma}
\begin{proof}
    If $G$ is a tree, the equality is rather trivial, so suppose $G$ is not a tree. In this case, let $r \in [m_1]$ be a right vertex in a cycle in $G$. In particular, there exist two distinct edges $e_1=l_1r$ and $e_2=l_2r$ such that $G_s=(V(G),E(G)\setminus\{e_s\})$ is still connected for $s=1,2$. Let $\bar{k}=k_{e_1}+k_{e_2}$. Then
    \begin{align*}
        W^k(G)=\sum_{w_1 \ne \cdots \ne w_{m_2}}\sum_{t_1 \ne \cdots \ne t_{m_1}} \prod_{s=1,2}\left(b_{w_{l_s}t_r}^{\bar{k}}\prod_{e=ij \ne e_1,e_2} b_{w_it_j}^{k_e}\right)^{k_{e_s}/\bar{k}}.
    \end{align*}
    Holder's Inequality implies that
    \begin{align*}
        W^k(G) \le \max_{s=1,2} \sum_{w_1 \ne \cdots \ne w_{m_2}}\sum_{t_1 \ne \cdots \ne t_{m_1}}b_{w_{l_s}t_r}^{\bar{k}}\prod_{e=ij \ne e_1,e_2} b_{w_it_j}^{k_e}=\max_{s=1,2} W^{k'_s}(G_s).
    \end{align*}
    Notice that $G_s$ runs over all vertices of $G$, $|E(G_s)|=|E(G)|-1$ and $G_i$ is still connected. Moreover, the neighborhood of $v \ne r$ is preserved and so are the weights for all $v \in [m_1]$, namely,
    \begin{align*}
        (k'_s)_v:=\sum_{u \in N(v,G_s)}(k'_s)_{(uv)}=\sum_{u \in N(v,G)}k_{(uv)}=k_v \ge 4,
    \end{align*}
    where $N(v,G)$ denotes the neighborhood of $v$ in $G$, and the last inequality follows as $v$ has at least two neighbors in $G$. The result follows by induction (see \cite[Lemma 2.9]{latala2018dimension}).
\end{proof}

Let $\mathcal{T}_{m_1,m_2}$ be the set of bipartite trees over $[m_2] \sqcup [m_1]$. By Lemma \ref{lemma reduction to trees}, we can assume that $G \in \mathcal{T}_{m_1,m_2}$. In \cite{latala2018dimension}, the authors developed a method to prune leaves of $G$ iteratively. In our case, however, we will prune the right vertices. To keep the notation clean, let
\begin{align*}
    W(G)=\sum_{w \in [d]^{m_2}_{\ne}}\sum_{t \in [n]^{m_1}_{\ne}} \prod_{e=ij \in E(G)}b_{w_it_j}^{(e)},
\end{align*}
where $(b^{(e)})_{e \in E(G)}$ is a family of $d\times n$ matrices and 
\begin{align*}
    [m]^I_{\ne}:=\{w \in [m]^I: w_k \ne w_l\, , \forall k \ne l \in I\}.
\end{align*}
We can easily recover $W^k(G)$ by setting $b^{(e)}_{wt}=b^{k_e}_{wt}$. 

We have the analog of Lemma 2.10 in \cite{latala2018dimension}. Let $\mathcal{L}(G)$ be the set of leaves of $G$ and for each $v \in \mathcal{L}(G) \cap [m_1]$, let $u_v$ be its only neighbor.
\begin{lemma}\label{lemma after the reduction}
    For any $G \in \mathcal{T}_{m_1,m_2}$ and $p_v \ge 1$ such that
    \begin{align*}
        \sum_{v \in [m_1]}\frac{1}{p_v}=1,
    \end{align*}
    we have
    \begin{align*}
        W(G) \le &\prod_{v \in \mathcal{L}(G)\cap [m_1]} \left\{\sum_{i \in [d]}\left(\sum_{j \in [n]}b_{ij}^{(u_vv)}\right)^{p_v}\right\}^{\frac{1}{p_v}} \times \\
        &\prod_{\substack{v \in \mathcal{L}(G)^c \cap [m_1]\\ u \in N(v) \cap \mathcal{L}(G)^c}} \left\{\sum_{i \in [d]}\left[\sum_{j \in [n]}b_{ij}^{(uv)}\prod_{a \in N(v) \setminus \{u\}}\left(\sum_{l \ne i}b_{lj}^{(av)}\right)\right]^{p_v}\right\}^{\frac{1}{p_v}\frac{1}{\alpha_{uv}}},
    \end{align*}
    where $\alpha_{uv}$ satisfies
    \begin{align*}
        \sum_{u \in N(v) \cap \mathcal{L}(G)^c} \frac{1}{\alpha_{uv}}=1,
    \end{align*}
    for all $v \in \mathcal{L}(G)^c \cap [m_1]$.
\end{lemma}
Before proving this result, we will use the following easier version. Let $u=u(v)$ be the choice $u \in N(v) \cap \mathcal{L}(G)^c$ that maximizes the second term in the bound, then the following holds.
\begin{coro}\label{holder type inequality coro for trees}
   For any $G \in \mathcal{T}_{m_1,m_2}$ and $p_v \ge 1$ such that
    \begin{align*}
        \sum_{v \in [m_1]}\frac{1}{p_v}=1,
    \end{align*}
    we have
    \begin{align*}
        W(G) \le &\prod_{v \in \mathcal{L}(G)\cap [m_1]} \left\{\sum_{i \in [d]}\left(\sum_{j \in [n]}b_{ij}^{(u_vv)}\right)^{p_v}\right\}^{\frac{1}{p_v}} \times \\
        &\prod_{v \in \mathcal{L}(G)^c \cap [m_1]} \left\{\sum_{i \in [d]}\left[\sum_{j \in [n]}b_{ij}^{(uv)}\prod_{a \in N(v) \setminus \{u\}}\left(\sum_{l \ne i}b_{lj}^{(av)}\right)\right]^{p_v}\right\}^{\frac{1}{p_v}}.
    \end{align*} 
\end{coro}
\begin{proof}[Proof of Lemma \ref{lemma after the reduction}]
The proof follows by induction. If $m_1=1$, then it is easy to check that $p_v=1$ and
\begin{align*}
    W(G) \le \sum_{i \in [d]}\sum_{j \in [n]}\prod_{a \in N(1')\setminus \{1\}}b_{ij}^{(11')} \left(\sum_{l\ne i}b_{lj}^{(a1')}\right).
\end{align*}
Therefore, if $|N(1')|>1$, $W(G)$ has the second form on the bound shown in the lemma. Otherwise, $|N(v)|=1$ and the bound has the first form. Hence, we can assume that $m_1>1$.

Let $\mathcal{L}=\mathcal{L}(G)$ and $v_1,v_2 \in [m_1]$ be such that 
\begin{align*}d(v_1,v_2)=\max_{r,r' \in [m_1]}d(r,r'), 
\end{align*}
where the distance is the graph distance. Therefore, both $v_1$ and $v_2$ have only one neighbor $u_1 \in N(v_1) \cap \mathcal{L}^c$ and $u_2 \in N(v_2) \cap \mathcal{L}^c$. This follows the argument shown in Proposition \ref{propo for the shapes and operator norm}. Let then $H$ be the subgraph generated by removing $v_1$, $v_2$ and all leaves $(N(v_1)\cup N(v_2))\cap \mathcal{L}$. Denote $H=(I \sqcup J, E(H))$. Then we have
\begin{align*}
W(G) \le \sum_{\substack{w \in [d]^I_{\ne}}}\sum_{\substack{t \in [n]^J_{\ne}}}&\left[\sum_{j \in [n]}b_{w_{u_1}j}^{(u_1v_1)}\prod_{a \in N(v_1)\setminus\{u_1\}}\left(\sum_{l \ne w_{u_1}}b_{lj}^{(av_1)}\right)\right] \times \\
&\left[\sum_{j \in [n]}b_{w_{u_2}j}^{(u_2v_2)}\prod_{a \in N(v_2)\setminus\{u_2\}}\left(\sum_{l \ne w_{u_2}}b_{lj}^{(av_2)}\right)\right]\prod_{e=ab \in E(H)}b_{w_at_b}^{(ab)},
\end{align*}
where we define
\begin{align*}
\prod_{a \in N(v)\setminus\{u\}}\left(\sum_{l \ne w_u}b_{lj}^{(av)}\right)=1,
\end{align*}
if $N(v)\setminus\{u\}=\varnothing$. Using Holder's Inequality, we can estimate
\begin{align}\label{induction step}
W(G) \le&
\left\{\sum_{\substack{w \in [d]^I_{\ne}}}\sum_{\substack{t \in [n]^J_{\ne}}}\left[\sum_{j \in [n]}b_{w_{u_1}j}^{(u_1v_1)}\prod_{a \in N(v_1)\setminus\{u_1\}}\left(\sum_{l\ne w_{u_1}}b_{lj}^{(av_1)}\right)\right]^{1+\frac{p_{v_1}}{p_{v_2}}}
\prod_{e=ab \in E(H)}b_{w_at_b}^{(ab)}\right\}^{\frac{p_{v_2}}{p_{v_1}+p_{v_2}}}\times \nonumber \\
&\left\{\sum_{\substack{w \in [d]^I_{\ne}}}\sum_{\substack{t \in [n]^J_{\ne}}}\left[\sum_{j \in [n]}b_{w_{u_2}j}^{(u_2v_2)}\prod_{a \in N(v_2)\setminus\{u_2\}}\left(\sum_{l\ne w_{u_2}}b_{lj}^{(av_2)}\right)\right]^{1+\frac{p_{v_2}}{p_{v_1}}}
\prod_{e=ab \in E(H)}b_{w_at_b}^{(ab)}\right\}^{\frac{p_{v_1}}{p_{v_1}+p_{v_2}}}.
\end{align}
Note that this inequality preserves the number of summations of right and left vertices, and also the homogeneity. Note also that if $v \in J$, the neighbors of $v$ in $H$ and $G$ are the same.

The induction will be based on inequality \eqref{induction step}. Suppose, for some $r>1$ that
\begin{align*}
W(G) \le \prod_{h=1}^H\left\{\sum_{\substack{w \in [d]^{I_h}_{\ne}}}\sum_{\substack{t \in [n]^{J_h}_{\ne}}}\left[\sum_{j \in [n]}b_{w_{u_h}j}^{(u_hv_h)}\prod_{a \in N(v_h)\setminus\{u_h\}}\left(\sum_{l \ne w_{u_h}}b_{lj}^{(av_h)}\right)\right]^{q_h}
\prod_{e=ab \in E(G_h)}b_{w_at_b}^{(ab)}\right\}^{\frac{1}{\alpha_h}},
\end{align*}
where $H<\infty$, $N(v)$ is the neighbor of $v$ in $G$,
\begin{enumerate}
\item For every $h$, $u_h \in I_h$, $v_h \notin J_h$ and $av_h \notin E(G_h)$ for every $a \in N(v_h)\setminus \{u_h\}$;
\item For every $h$, $G_h$ is a tree over $I_h \sqcup J_h$ and $|J_h|=r$;
\item The inequality is $1$-homogeneous in all the variables $b^{(e)}$ and it preserves the number of left and right summations; 
\item The exponents $q_h$ satisfies
\begin{align*}
    q_h=\sum_{v \in [m_1]\setminus J_h}\frac{p_{v_h}}{p_{v}},
\end{align*}
and $\alpha_h \ge 1$.
\end{enumerate}
We aim to show that if this holds for $r>1$, so does it for $r-1$. Indeed, fix one of the terms
\begin{align*}
    T_h:=\left\{\sum_{\substack{w \in [d]^{I_h}_{\ne}}}\sum_{\substack{t \in [n]^{J_h}_{\ne}}}\left[\sum_{j \in [n]}b_{w_{u_h}j}^{(u_hv_h)}\prod_{a \in N(v_h)\setminus\{u_h\}}\left(\sum_{l\ne w_{u_h}}b_{lj}^{(av_h)}\right)\right]^{q_h}
\prod_{e=ab \in E(G_h)}b_{w_at_b}^{(ab)}\right\}.
\end{align*}
Since $G_h$ is a tree and $r>1$, there exists $r_h$ such that $u_h$ is not a leaf of $r_h$ and $r_h$ has only one neighbor $l_h$ such that $N(l_h)>1$ in $G_h$. Let then $H_h$ be the subgraph (a tree) of $G_h$ where we remove $r_h$ and all of its leaves and let $H_h=(I_{h'}\sqcup J_{h'},E(H_h))$, then
\begin{align*}
T_h\le \sum_{\substack{w \in [d]^{I_{h'}}_{\ne}}}\sum_{\substack{t \in [n]^{J_{h'}}_{\ne}}}&\left[\sum_{j \in [n]}b_{w_{u_h}j}^{(u_hv_h)}\prod_{a \in N(v_h)\setminus\{u_h\}}\left(\sum_{l\ne w_{u_h}}b_{lj}^{(av_h)}\right)\right]^{q_h}\times \\ 
&\left[\sum_{j \in [n]}b_{w_{l_h}j}^{(l_hr_h)}\prod_{a \in N(r_h)\setminus\{l_h\}}\left(\sum_{l\ne w_{l_h}}b_{lj}^{(ar_h)}\right)\right]\prod_{e=ab \in E(G_h)}b_{w_at_b}^{(ab)}.
\end{align*}
We can thus estimate by Holder's Inequality that
\begin{align*}
T_h \le &\left\{\sum_{\substack{w \in [d]^{I_{h'}}_{\ne}}}\sum_{\substack{t \in [n]^{J_{h'}}_{\ne}}}\left[\sum_{j \in [n]}b_{w_{u_h}j}^{(u_hv_h)}\prod_{a \in N(v_h)\setminus\{u_h\}}\left(\sum_{l\ne w_{u_h}}b_{lj}^{(av_h)}\right)\right]^{q_h'}\prod_{e=ab \in E(G_h)}b_{w_at_b}^{(ab)}\right\}^{1/\alpha_h}\times \\
&\left\{\sum_{\substack{w \in [d]^{I_{h'}}_{\ne}}}\sum_{\substack{t \in [n]^{J_{h'}}_{\ne}}}\left[\sum_{j \in [n]}b_{w_{l_h}j}^{(l_hr_h)}\prod_{a \in N(r_h)\setminus\{l_h\}}\left(\sum_{l\ne w_{l_h}}b_{lj}^{(ar_h)}\right)\right]^{q}\prod_{e=ab \in E(G_h)}b_{w_at_b}^{(ab)}\right\}^{1/q},
\end{align*}
where $q_h'/q_h$ and $q$ are conjugate exponents. Again, the inequality is $1$-homogeneous in all the variables it involves, and it preserves the number of summations. Moreover, we can set
\begin{align*}
&q_h'=\sum_{v \in [m_1]\setminus J_{h'}}\frac{p_{v_h}}{p_v}\\
&q=\sum_{v \in [m_1]\setminus J_{h'}}\frac{p_{l_h}}{p_v},
\end{align*}
and it is easy to check that indeed $q_h'/q_h$ and $q$ are conjugate exponents. Note that each new term has the same form as in the induction step with $|J_{h'}|=r-1$, therefore the induction is proved. 

The previous argument also shows that the induction holds for $r=0$. Since the choice of $u \in \mathcal{L}^c \cap N(v)$ is arbitrary for each $v$, we deduce
\begin{align*}
W(G) &\le \prod_{h=1}^H\left\{\sum_{\substack{w \in [d]^{I_{h}}_{\ne}}}\sum_{\substack{t \in [n]^{J_{h}}_{\ne}}}\left[\sum_{j \in [n]}b_{w_{u_h}j}^{(u_hv_h)}\prod_{a \in N(v_h)\setminus\{u_h\}}\left(\sum_{l\ne w_{u_h}}b_{lj}^{(av_h)}\right)\right]^{q_h}
\prod_{e=ab \in E(G_h)}b_{w_at_b}^{(ab)}\right\}^{\frac{1}{\alpha_h}}\\
&\le \prod_{e=uv \in E(G):u \in \mathcal{L}^c}\left\{\sum_{i \in [d]}\left[\sum_{j \in [n]}b_{ij}^{(e)}\prod_{a \in N(v)\setminus\{u\}}\left(\sum_{l \ne i}b_{lj}^{(av)}\right)\right]^{p_{v}}\right\}^{\frac{1}{\alpha_e}}.
\end{align*}
The conclusion of the lemma follows by the renormalization $\alpha_{uv} \leftarrow p_v\alpha_{uv}$ and splitting the product over $v \in \mathcal{L}$ and $v \notin \mathcal{L}$.
\end{proof}

Now we can prove Proposition \ref{propo for the shapes and the schatten norm}.
\begin{proof}[Proof of Proposition \ref{propo for the shapes and the schatten norm}]
  Let $|k|=\sum_v k_v=2p$ and $\mathcal{L}=\mathcal{L}(G) \cap [m_1]$. By Lemma \ref{lemma reduction to trees} and Corollary \ref{holder type inequality coro for trees} with $p_v=|k|/k_v$, we get
  \begin{align*}
      W^k(G) \le W^{k'}(T)\le  &\prod_{v \in \mathcal{L}}\left\{\sum_{i \in [d]}\left(\sum_{j \in [n]}b_{ij}^{k_{v}}\right)^{\frac{|k|}{k_v}}\right\}^{\frac{k_v}{|k|}} \times \\
&\prod_{v \in \mathcal{L}^c}\left\{\sum_{i \in [d]}\left[\sum_{j \in [n]}b_{ij}^{k_{uv}}\prod_{a \in N(v) \setminus\{u\}}\left(\sum_{l\ne i}b_{lj}^{k_{av}}\right)\right]^{\frac{|k|}{k_v}}\right\}^{\frac{k_v}{|k|}},
  \end{align*}
  where $T$ is the spanning tree of $G$ that maximizes $W^{k'}(T')$ in Lemma \ref{lemma reduction to trees}.  Since $k_v \ge 4$ and $b_{ij}^{k_{uv}} \le b_{ij}^2\sigma_*^{k_{uv}-2}$, we get
\begin{align*}
    W^k(G) \le &\sigma_*^{|k|-4|\mathcal{L}|-2\sum_{v \in \mathcal{L}^c}|N(v)|}\prod_{v \in \mathcal{L}}\left\{\sum_{i \in [d]}\left(\sum_{j \in [n]}b_{ij}^{4}\right)^{\frac{|k|}{k_v}}\right\}^{\frac{k_v}{|k|}} \times \\
&\prod_{v \in \mathcal{L}^c}\left\{\sum_{i \in [d]}\left[\sum_{j \in [n]}b_{ij}^{2}\prod_{a \in N(v) \setminus\{u\}}\left(\sum_{l\ne i}b_{lj}^{2}\right)\right]^{\frac{|k|}{k_v}}\right\}^{\frac{k_v}{|k|}}.
\end{align*}
As $T$ is a spanning tree, we have 
\begin{align*}
    &\sum_{v \in \mathcal{L}^c}|N(v)|+|\mathcal{L}|=m_2+m_1-1;\\
    &\sum_{v \in \mathcal{L}^c}|N(v)|-2|\mathcal{L}^c|=m_2-m_1-1+|\mathcal{L}|.
\end{align*}
Moreover, we can remove $\sigma_C$ from each term in the second product to get that
\begin{align*}
    W^k(G) \le &\sigma_*^{|k|-2(m_1+m_2-1)-2|\mathcal{L}|}\sigma_C^{2(m_1+m_2-1)+2|\mathcal{L}|}\prod_{v \in \mathcal{L}}\left\{\sum_{i \in [d]}\left(\sum_{j \in [n]}b_{ij}^{4}\right)^{\frac{|k|}{k_v}}\right\}^{\frac{k_v}{|k|}} \times \\
&\prod_{v \in \mathcal{L}^c}\left\{\sum_{i \in [d]}\left[\sum_{j \in [n]}b_{ij}^{2}\left(\sum_{l\ne i}b_{lj}^{2}\right)\right]^{\frac{|k|}{k_v}}\right\}^{\frac{k_v}{|k|}}.
\end{align*}
Finally, the inequality of the norms in $\real^d$ implies that
\begin{align*}
    \norm{\cdot}_{\frac{|k|}{k_v}} \le d^{\frac{k_v-4}{|k|}}\norm{\cdot}_{\frac{|k|}{4}},
\end{align*}
so we deduce
\begin{align*}
    W^k(G) \le d\sigma_*^{|k|-2(m_1+m_2-1)-2|\mathcal{L}|}\sigma_C^{2(m_1+m_2-1)+2|\mathcal{L}|} \sigma_p^{2|\mathcal{L}|} \bar{\sigma}_p^{2|\mathcal{L}^c|}.
\end{align*}
The proof of Proposition \ref{propo for the shapes and the schatten norm} follows by a straightforward computation and the fact that $0 \le |\mathcal{L}| \le m_1$.
\end{proof}
\begin{remark}\label{Remark about l diferent than i}
    Note that we rather proved Theorem \ref{thm second main} with a parameter $\sigma_p'$ instead of $\sigma_p$, where $\sigma_p'$ only takes $l \ne i$, that is,
    \begin{align*}
        \sigma_p'=\left\{\sum_{i \in [d]}\left[\sum_{j \in [n]}\sum_{l \ne i}b_{ij}^2b_{lj}^2\right]^{p/2}\right\}^{1/p}.
    \end{align*}
    This minor change is only important for cases where the contribution of a column $X_j$ appears only in the diagonal part, that is, when $X_j=b_{ij}e_{i}$ for some $i$. 
\end{remark}


\section{Examples}\label{sec examples}

Let us start by recalling the previous known results in \cite{cai2022non,bandeira2021matrix}.
\begin{thm}[Theorem 2.1 in \cite{cai2022non}]\label{thm:zhang}
    In the setting of Theorem \ref{main theorem}, we have
    \begin{align*}
        &\esp \norm{XX^T-\esp XX^T} \\
        &\le (1+\varepsilon)\left\{2\sigma_R\sigma_C+\sigma_C^2+C(\varepsilon)(\sigma_C\sigma_*+\sigma_R\sigma_*)\sqrt{\log(n\wedge d)}+C^2(\varepsilon)\sigma_*^2\log(n\wedge d)\right\}.
    \end{align*}
\end{thm}

\begin{thm}[Theorem 3.12 in \cite{bandeira2021matrix}]\label{thm:bandeira}
    Let $X$ be a $d\times n$ Gaussian matrix with independent entries and $X_{ij}=b_{ij}g_{ij}$, for $b_{ij} \ge 0$. Then 
    \begin{align*}
        &\esp \norm{XX^T-\esp XX^T} \\
        &\le \norm{X_{\emph{free}}X_{\emph{free}}^T-\esp XX^T \otimes 1}+C\left\{\sigma(X)\tilde{v}(X) \log^{3/4}(nd)+\tilde{v}^2(X) \log ^{3/2}(nd)\right\}.
    \end{align*}
\end{thm}
\begin{coro}
    In the setting of Theorem \ref{thm:bandeira}, we have
    \begin{align*}
        \norm{X_{\emph{free}}X_{\emph{free}}^T-\esp XX^T \otimes 1} &\le 2\max_{i \in [d]}\left(\sum_{j \in [n]}\sum_{l \in [d]}b_{ij}^2b_{lj}^2\right)^{1/2}+\sigma_C^2\\
        &=2\sigma_\infty +\sigma_C^2
    \end{align*}
    and
    \begin{align*}
        &\sigma(X)= \max(\sigma_C,\sigma_R);\\
        &\tilde{v}(X)^2 \asymp \sigma_* \sigma(X).
    \end{align*}
    Therefore, we have
    \begin{align}\label{thm of free ramon}
        &\esp \norm{XX^T-\esp XX^T} \nonumber \\
        &\le 2\sigma_\infty+\sigma_C^2+C\left[\sigma_*^{1/2}\sigma_C^{3/2}+\sigma_*^{1/2}\sigma_R^{3/2}\right]\log^{3/4}(nd)+C\left[ \sigma_*\sigma_C+\sigma_*\sigma_R\right]\log^{3/2}(nd).
    \end{align}
\end{coro}
\begin{proof}
    Let $X_j$ be the $j$th column of $X$. Then
    \begin{align*}
        \Sigma_j=\esp X_jX_j^T=\text{diag}(b_{ij}^2).
    \end{align*}
    Hence
    \begin{align*}
        &\left\|\sum_{j \in [n]}\Sigma_j\right\|=\max_{i \in [d]}\sum_{j \in [n]}b_{ij}^2=\sigma_R^2;\\
        &\max_{j \in [n]}\text{Tr}(\Sigma_j)=\max_{j \in [n]}\sum_{i \in [d]}b_{ij}^2=\sigma_C^2.
    \end{align*}
    The computation for the parameters $\tilde{v}(X)$ and $\sigma(X)$ then follows by Lemma 3.8 in \cite{bandeira2021matrix}.

    On the other hand, denote
    \begin{align*}
        X=\sum_{i,j}g_{ij}b_{ij}E_{ij}=\sum_{k}g_{k}A_{k},
    \end{align*}
    where $E_{ij}$ is the canonical basis of the space of $d\times n$ matrices. Then the authors of \cite{bandeira2021matrix} computed that $X_{\text{free}}=U+V$, where
    \begin{align*}
        &U=\sum_{k}A_{k}\otimes l(e_k);\\
        &V=\sum_{k}A_k \otimes l^*(e_k),
    \end{align*}
    and $l$ is the creation operator of the free Fock space over $\mathbb{C}^{nd}$. In particular, we have $l^*(e_k)l(e_j)=\delta_{kj}1$. Therefore,
    \begin{align*}
        VV^*=\sum_k A_kA_k^* \otimes 1=\esp XX^T \otimes 1;\\
        U^*U=\sum_{k} A_k^*A_k \otimes 1=\esp X^TX \otimes 1.
    \end{align*}
    Hence they deduced that
    \begin{align*}
        \norm{X_{\text{free}}X_{\text{free}}^T-\esp XX^T \otimes 1} \le \norm{UV^*+VU^*+UU^*} \le 2\norm{UV^*}+\norm{UU^*}.
    \end{align*}
    The second one follows easily as
    \begin{align*}
        \norm{UU^*}=\norm{\esp X^TX}=\sigma_C^2.
    \end{align*}
    On the other hand, in \cite{bandeira2021matrix}, they used $\norm{UV^*} \le \norm{U}\norm{V}=\sigma_R\sigma_C$ to bound the first term. However, in the case of independent entries, it is a straightforward computation to check that
    \begin{align*}
        \norm{UV^*}=\sigma_\infty,
    \end{align*}
    and the result follows.
\end{proof}

Now we discuss various examples and present how Theorem \ref{main theorem} improves upon Theorems \ref{thm:zhang} and \ref{thm:bandeira}.
\begin{exa}
Assume the columns of $X$ are i.i.d, namely, $b_{ij}=b_i$. In this case, we have
\begin{align*}
    &\bullet \sigma_C=\norm{b}_2;\\
    &\bullet \sigma_R=\sqrt{n} \norm{b}_\infty;\\
    &\bullet \sigma_*=\norm{b}_\infty; \\
    &\bullet \tilde{\sigma}_\infty \le \sqrt{n}\norm{b}_\infty^2;\\
    &\bullet \bar{\sigma}_\infty=\sqrt{n}\norm{b}_\infty^2;\\
    &\bullet \sigma_\infty \le \sqrt{n}\norm{b}_\infty\norm{b}_2.
\end{align*}
In particular, 
\begin{align*}
    \frac{\tilde{\sigma}_\infty \sigma_C}{\sigma_*} \le \sqrt{n}\norm{b}_\infty\norm{b}_2.
\end{align*}
Hence, both bounds shown in Theorem \ref{main theorem} yield 
\begin{align*}
    \esp \norm{XX^T-\esp XX^T} \lesssim &
    \sqrt{n}\norm{b}_2\norm{b}_\infty+\norm{b}_2^2+\\
    &C\left[\norm{b}_2\norm{b}_\infty+\sqrt{n}\norm{b}_\infty^2\right]\sqrt{\log(n\wedge d)}+\\
    & C\norm{b}_\infty^2 \log(n\wedge d).
\end{align*}
The leading term agrees with the sharp bound derived in \cite{cai2022non,koltchinskii2017concentration}. 
\end{exa}

\begin{exa}
    Let $X$ be a Gaussian matrix with i.i.d rows, that is, $b_{ij}=b_{j}$. In this case, we have
    \begin{align*}
        &\bullet \sigma_C=\sqrt{d}\norm{b}_\infty;\\
        &\bullet \sigma_R =\norm{b}_2;\\
        &\bullet \sigma_*=\norm{b}_\infty;\\
        &\bullet \sigma_\infty=\sqrt{d-1} \norm{b}_4^2; \\
        &\bullet \tilde{\sigma}_\infty =\bar{\sigma}_\infty=\norm{b}_4^2.
    \end{align*}
    In particular,
    \begin{align*}
        \frac{\tilde{\sigma}_\infty \sigma_C}{\sigma_*}=\sqrt{d}\norm{b}_4^2.
    \end{align*}
    Hence, Theorem \ref{main theorem} implies that
    \begin{align*}
        \esp\norm{XX^T-\esp XX^T} \lesssim &\sqrt{d} \norm{b}_4^2+d\norm{b}_\infty^2 +\\
        &C\left[\sqrt{d}\norm{b}_\infty^2+\norm{b}_4^2\right]\sqrt{\log(n\wedge d)}+\\
        &C\norm{b}_\infty^2\log(n\wedge d).
    \end{align*}
    In this case, the error factor is smaller than the leading one, hence
    \begin{align*}
       \esp\norm{XX^T-\esp XX^T} \lesssim &\sqrt{d} \norm{b}_4^2+d\norm{b}_\infty^2.
    \end{align*}
    This agrees with the sharp result in \cite{cai2022non}. However, they had to derive a different method to prove this case, whereas we deduce directly from our main result that covers all cases.

    Moreover, \eqref{thm of free ramon} gives
    \begin{align*}
        \esp\norm{XX^T-\esp XX^T} \lesssim &\sqrt{d}\norm{b}_4^2+d\norm{b}_\infty^2+\\
        &C\left[\norm{b}_2^{3/2}\norm{b}_\infty^{1/2}+d^{3/4}\norm{b}_\infty^2\right]\log^{3/4}(nd)+\\
        &C\left[\sqrt{d}\norm{b}_\infty+\norm{b}_2\right]\norm{b}_\infty\log^{3/2}(nd).
    \end{align*}
    Here, we observe that the error factor is not necessarily smaller than the leading one.
\end{exa}

\begin{exa}
    Consider $b_{ij}=a_ib_j$. Then
    \begin{align*}
        &\bullet \sigma_C=\norm{a}_2 \norm{b}_\infty;\\
        &\bullet \sigma_R=\norm{a}_\infty \norm{b}_2;\\
        &\bullet \sigma_*=\norm{a}_\infty \norm{b}_\infty;\\
        &\bullet \tilde{\sigma}_\infty \le \norm{b}_4^2\norm{a}_\infty^2;\\
        &\bullet \bar{\sigma}_\infty =\norm{b}_4^2\norm{a}_\infty^2;\\
        &\bullet \sigma_\infty \le \norm{b}_4^2 \norm{a}_2\norm{a}_\infty.
    \end{align*}
    We observe that
    \begin{align*}
        \frac{\tilde{\sigma}_\infty\sigma_C}{\sigma_*} \le \norm{b}_4^2\norm{a}_2\norm{a}_\infty.
    \end{align*}
    Therefore, Theorem \ref{main theorem} implies that
    \begin{align*}
        \esp \norm{XX^T-\esp XX^T} \lesssim & \norm{b}_4^2\norm{a}_2\norm{a}_\infty +\norm{a}_2^2\norm{b}_\infty^2+\\
        &C\left[\norm{a}_2\norm{a}_\infty \norm{b}_\infty^2+\norm{b}_4^2\norm{a}_\infty^2\right]\sqrt{\log (n\wedge d)}+\\
        &C\norm{a}_\infty^2\norm{b}_\infty^2 \log(n\wedge d).
    \end{align*}
    The result in \cite{cai2022non} yields
    \begin{align*}
        \esp \norm{XX^T-\esp XX^T} \lesssim & \norm{b}_2\norm{b}_\infty\norm{a}_2\norm{a}_\infty+\norm{a}_2^2\norm{b}_\infty^2+\\
        &C\left[\norm{a}_2\norm{a}_\infty \norm{b}_\infty^2+\norm{b}_2\norm{b}_\infty\norm{a}_\infty^2\right]\sqrt{\log (n\wedge d)}+\\
        &C\norm{a}_\infty^2\norm{b}_\infty^2 \log(n\wedge d).
    \end{align*}
    And finally, \eqref{thm of free ramon} gives
    \begin{align*}
        \esp \norm{XX^T-\esp XX^T} \lesssim & \norm{b}_4^2\norm{a}_2\norm{a}_\infty +\norm{a}_2^2\norm{b}_\infty^2+\\
        &C\left[\norm{a}_2^{3/2}\norm{a}_\infty^{1/2}\norm{b}_\infty^2+\norm{b}_2^{3/2}\norm{b}_\infty^{1/2}\norm{a}_\infty^2\right]\log^{3/4}(nd)+\\
        &C\left[\norm{a}_2\norm{b}_\infty+\norm{b}_2\norm{a}_\infty\right]\norm{a}_\infty\norm{b}_\infty \log^{3/2}(nd).
    \end{align*}
    In this case, Theorem \ref{main theorem} strictly improves both and sheds light on the 4th-moment appearing for $b_j$.
\end{exa}

Our final example is where all columns have approximately the same norm.
\begin{exa}
    Suppose there exists $K \ge 1$ such that 
    \begin{align*}
        \frac{1}{K}\norm{b_k}_2 \le \norm{b_j}_2 \le K\norm{b_k}_2,
    \end{align*}
    for all $k,l \in [n]$. Then it is easy to compute 
    \begin{align*}
        \beta_\infty \le K.
    \end{align*}
    By Theorem \ref{main theorem}, we have
    \begin{align*}
        &\esp \norm{XX^T-\esp XX^T} \\
        &\le (1+\varepsilon)\left\{
    2K\sigma_\infty+\sigma_C^2+C(\varepsilon)\sigma_*\left(\sigma_C+\sigma_R\right)\sqrt{\log(n \wedge d)}+C^2(\varepsilon)\sigma_*^2 \log (n\wedge d)
    \right\}.
    \end{align*}
    Here, we do not require any additional structure on $B$ and the previous known results only show the leading term with $\sigma_C\sigma_R$ in \cite{cai2022non} and a large error factor in \cite{bandeira2021matrix}.
\end{exa}





\section{Lower bounds}\label{section lower bounds}
We first begin the lower bounds for the $p$-moment of the Schatten norm.
\begin{propo}
For any even $p \ge 2$ and $X$ satisfying the assumptions in Theorem \ref{main theorem}, we have
\begin{align*}
    \left(\esp \emph{Tr}(XX^T-\esp XX^T)^p\right)^{1/p} \gtrsim \sigma_p+\sigma_C^2+\sqrt{p}\bar{\sigma}_p+pb_p^2.
\end{align*}
\end{propo}
\begin{proof}
    By Lemma \ref{joint moments Gaussians lemma}, the joint moments of $g$ and $g^2-1$ are always positive, thus it follows that
    \begin{align*}
        \left(\esp \text{Tr}(XX^T-\esp XX^T)^p\right)^{1/p} \ge \left(\esp \text{Tr}(\text{Diag}(XX^T)-\esp XX^T)^p\right)^{1/p} \gtrsim \sqrt{p}\bar{\sigma}_p+pb_p^2.
    \end{align*}
For the leading factor, note that the Schatten norm is always lower bounded by the mixed $l_2(l_p)$ norm (see Lemma 2.12 in \cite{latala2018dimension}), then Jensen's Inequality implies that
\begin{align*}
    \left(\esp \text{Tr}(XX^T-\esp XX^T)^p\right)^{1/p} &\ge \left[\sum_{i \in [d]}\left(\sum_{l\in [d]}\esp (XX^T-\esp XX^T)^2_{il}\right)^{p/2}\right]^{1/p}.\\
\end{align*}
The latter can be estimated as
\begin{align*}
    \esp (XX^T-\esp XX^T)_{il}^2&=\sum_{j \in [n]} \esp (X_jX_j^T-\esp X_jX_j)^2_{il}\\
    &=\sum_{j \in [n]} \esp \left(b_{ij}b_{lj}(g_{ij}g_{lj}-\mathbf{1}_{i=l})\right)^2\\
    &\ge \sum_{j \in [n]}b_{ij}^2b_{lj}^2.
\end{align*}
Hence
\begin{align*}
    \left(\esp \text{Tr}(XX^T-\esp XX^T)^p\right)^{1/p} &\ge\left[\sum_{i \in [d]}\left(\sum_{j \in [n]}\sum_{l\in [d] }b_{ij}^2b_{lj}^2\right)^{p/2}\right]^{1/p}\\
    &\ge\sigma_p.
\end{align*}

Finally, let $j^*$ be the column with the largest Euclidean norm, that is, $\norm{b_{j^*}}_2=\sigma_C$, then
\begin{align*}
     \left(\esp \text{Tr}(XX^T-\esp XX^T)^p\right)^{1/p} \ge  \left(\esp \text{Tr}(X_{j^*}X_{j^*}^T-\esp X_{j^*}X_{j^*}^T)^p\right)^{1/p}.
\end{align*}
The sharp result for the i.i.d case proved by  Koltchinskii and Lounici in \cite{koltchinskii2017concentration} implies that
\begin{align*}
    \left(\esp \text{Tr}(X_{j^*}X_{j^*}^T-\esp X_{j^*}X_{j^*}^T)^p\right)^{1/p} & \ge \esp \norm{X_{j^*}X_{j^*}^T-\esp X_{j^*}X_{j^*}^T}\\
    &\gtrsim \sigma_C^2,
\end{align*}
where we assume $p$ is even for the first inequality.
\end{proof}

For the operator norm, we have the following result.
\begin{lemma}
    Let $X$ be a random matrix satisfying the assumptions of Theorem \ref{main theorem}. Then
    \begin{align*}
        \left(\esp \norm{XX^T-\esp XX^T}^2\right)^{1/2} \gtrsim \sigma_\infty+\sigma_C^2.
    \end{align*}
\end{lemma}
\begin{proof}
    Let $S_j=X_jX_j^T-\esp X_iX_i^T$, then Tropp's result in \cite{tropp2016expected} implies that
\begin{align*}
    (\esp \norm{XX^T-\esp XX^T}^2)^{1/2} \gtrsim \left\|\sum_{j \in [n]}\esp S_jS_j^T\right\|^{1/2}+ (\esp \max_{j \in [n]} \norm{S_j}^2)^{1/2}.
\end{align*}
The matrix in the first bound can then be computed as
\begin{align*}
    \sum_{j \in [n]}\esp X_jX_j^TX_jX_j^T-(\esp X_jX_j^T)^2.
\end{align*}
This can easily be seen as a diagonal matrix (a similar argument was proved in \cite{cai2022non}) and lower bounded by
\begin{align*}
    \left\|\sum_{j \in [n]}\esp X_jX_j^TX_jX_j^T-(\esp X_jX_j^T)^2\right\| \ge \sigma_\infty^2,
\end{align*}
hence
\begin{align*}
    (\esp \norm{XX^T-\esp XX^T}^2)^{1/2} \gtrsim \sigma_\infty.
\end{align*}
On the other hand, the second term can be bounded as
\begin{align*}
    (\esp \max_{j \in [n]} \norm{S_j}^2)^{1/2} \ge \max_{j \in [n]}(\esp \norm{S_j}^2)^{1/2} \gtrsim \sigma_C^2,
\end{align*}
where we again use the lower bound of \cite{koltchinskii2017concentration}.

\end{proof}




\printbibliography

\end{document}